\documentclass{amsart}
\usepackage[sorted-cites]{amsrefs}
\usepackage{amsfonts}
\usepackage{graphicx}
\usepackage{amscd}
\usepackage{amsmath}
\usepackage{amssymb}
\usepackage{xcolor}

\makeatletter
\@namedef{subjclassname@2010}{%
  \textup{2010} Mathematics Subject Classification}
\makeatother

\usepackage{mathrsfs}
\usepackage{amsbsy}

\newtheorem{lemma}{\bf Lemma}

\newtheorem{proposition}{\bf Proposition}
\newtheorem{remark}{Remark}

\newtheorem{theorem}{\bf Theorem}

\theoremstyle{definition}
\newtheorem{example}{\bf Example}

\numberwithin{equation}{section}

\makeatletter
\@namedef{subjclassname@2020}{%
  \textup{2020} Mathematics Subject Classification}
\makeatother

\title[Gradient shrinking Ricci solitons]{Four-dimensional complete gradient \\shrinking Ricci solitons}

\author{Huai-Dong Cao}
\author{Ernani Ribeiro Jr}
\author{Detang Zhou}

\address[H.-D. Cao]{Department of Mathematics, Lehigh University, Bethlehem, PA 18015, USA} \email{huc2@lehigh.edu} 
\thanks{H.-D. Cao was partially supported by a Simons Foundation Collaboration Grant (\# 586694 HC)}

\address[E. Ribeiro Jr]{Departamento  de Matem\'atica, Universidade Federal do Cear\'a - UFC, Campus do Pici, 60455-760, Fortaleza - CE, Brazil.}
\email{ernani@mat.ufc.br}
\thanks{E. Ribeiro was partially supported by CNPq/Brazil (\# 305410/2018-0 \& 160002/2019-2) and CAPES/ Brazil - Finance Code 001}

\address[D. Zhou]{Instituto de Matem\'atica e Estat\'istica, Universidade Federal Fluminense - UFF, 24020-140, Niter\'oi - RJ, Brazil}
\email{zhou@impa.br}
\thanks{D. Zhou was partially supported by CNPq/Brazil, (\# 305364/2019-7) and FAPERJ/ Brazil (\# E-26/202.591/2019). }

\date{July 2, 2020}

\begin{document}

\begin{abstract}
In this article, we study four-dimensional complete gradient shrinking Ricci solitons. We prove that a four-dimensional complete gradient shrin\-king Ricci soliton satisfying a pointwise condition involving either the self-dual or anti-self-dual part of the Weyl tensor is either Einstein, or a finite quotient of either the Gaussian shrinking soliton $\Bbb{R}^4,$ or  $\Bbb{S}^{3}\times\Bbb{R}$, or  $\Bbb{S}^{2}\times\Bbb{R}^{2}.$ In addition, we provide some curvature es\-ti\-mates for four-dimensional complete gradient Ricci solitons assuming that its scalar curvature is suitable bounded by the potential function. 
\end{abstract}

\maketitle

\section{Introduction}\label{int}

A complete Riemannian metric $g$ on a smooth manifold $M^n$ is called a {\it gradient shrinking Ricci soliton} if there exists a smooth potential function $f$ on $M^n$ such that the Ricci tensor $Ric$ of the metric $g$ satisfies the equation
\begin{equation}
\label{maineq}
Ric+ Hess\,f=\frac{1}{2} g.
\end{equation} Here, $Hess\,f$ denotes the Hessian of $f.$ Gradient Ricci solitons are important objects in understanding Hamilton's Ricci flow \cite{Hamilton1}. They often arise as Type I singularity models \cite{Hamilton2}, thus playing a crucial role in the singularity analysis of the Ricci flow \cite{Perelman2}. Indeed, it was conjectured by Hamilton and confirmed by Enders, M\"uller and Topping \cite{Topping} (see also \cite{Naber, Sesum} for related works) that, under certain mild restriction, the blow-ups around a Type I singularity point of the Ricci flow converge to (nontrivial) gradient shrinking Ricci solitons.  We refer the reader to the survey \cite{caoALM11} and references therein for an overview on the subject. 

In dimension $n=2$, Hamilton \cite{Hamilton2} showed that any two-dimensional gradient shrinking Ricci soliton is either isometric to the plane $\Bbb{R}^2$ or a quotient of the sphere $\Bbb{S}^2.$ For $n=3$, by the works of Ivey \cite{Ivey}, Perelman \cite{Perelman2}, Naber \cite{Naber}, Ni-Wallach \cite{Ni}, and Cao-Chen-Zhu \cite{CaoA},  it is known that any complete three-dimensional gradient shrinking Ricci soliton is a finite quotient of either the round sphere $\Bbb{S}^3,$ or the Gaussian shrinking soliton $\Bbb{R}^3,$ or the round cylinder $\Bbb{S}^{2}\times\Bbb{R}.$  In recent years, a lot of progress has been made in understanding gradient shrinking Ricci solitons in higher dimensions; see, e.g., \cite{CaoChen,ELM, Ni, Chen,Chow,zhang, Naber, CZ, PW2, CWZ, FLGR, MW,KW, MS, MW2} and the references therein.

Before discussing the four-dimensional case, we emphasize that dimension four displays fascinating and peculiar features, for this reason much attention has been given to this dimension.  Many peculiar features on oriented 4-manifolds directly rely on the fact that the bundle of $2$-forms $\Lambda^2$ can be invariantly decomposed as a direct sum 
\begin{equation}
\label{piu}
\Lambda^2=\Lambda^{+}\oplus\Lambda^{-},
\end{equation}  
where $\Lambda^{\pm}$ is the $(\pm 1)$-eigenspace of the Hodge star operator $\ast,$ respectively. In this scenario, we now briefly recall a few relevant results on the classification of four-dimensional complete gradient shrinking Ricci solitons under vanishing conditions involving the Weyl tensor $W.$ For instance, it is known by the works of \cite{ELM, Ni, zhang, PW2, CWZ, MS} that complete locally conformally flat (i.e. $W=0$) four-dimensional gradient shrinking Ricci solitons are isometric to finite quotients of either $\Bbb{S}^4$, $\Bbb{R}^4$, or $\Bbb{S}^{3}\times\Bbb{R}.$ Later, Chen and Wang \cite{CW} proved that half-conformally flat (i.e. $W^{+}=0$ or $W^{-}=0$) four-dimensional gradient shrinking Ricci solitons are isometric to finite quotients of $\Bbb{S}^4$, $\Bbb{CP}^2,$ $\Bbb{R}^4$, or $\Bbb{S}^{3}\times\Bbb{R}.$ In \cite{CaoChen}, Cao and Chen were able to prove that complete Bach-flat four-dimensional gradient shrinking Ricci solitons are either Einstein or finite quotients of $\Bbb{R}^4$ or $\Bbb{S}^{3}\times\Bbb{R}.$ We point out that either locally conformally flat metrics, Einstein metrics, or half-conformally flat metrics on four-dimensional manifolds are Bach-flat. On the other hand, Fern\'{a}ndez-L\'opes and Garc\'ia-R\'io \cite{FLGR} together with Munteanu-Sesum \cite{MS} showed that complete gradient shrinking Ricci solitons with harmonic Weyl tensor (i.e. $\delta W=0$) are rigid, namely, they are either Einstein, or a finite quotient of $\Bbb{R}^4$, $\Bbb{S}^{2}\times\Bbb{R}^{2}$ or $\Bbb{S}^{3}\times\Bbb{R}.$ A more recent result, due to Wu, Wu and Wylie \cite{Wu},  states that the same conclusion holds under the weaker condition of harmonic self-dual Weyl tensor (i.e. $\delta W^{+}=0$). 
In \cite{catinoAdv}, Catino showed that every four-dimen\-sio\-nal compact shrinking Ricci soliton satisfying the integral pinching condition $$\int_{M}|W|^{2} dV_{g}+\frac{5}{4}\int_{M}|\mathring{Ric}|^{2}dV_{g}\leq \frac{1}{48}\int_{M}R^{2} dV_{g}$$ is isometric to a quotient of $\Bbb{S}^{4}.$ Here, $\mathring{Ric}$ stands for the traceless Ricci tensor. Other interesting results on four-dimen\-sional compact shrinking Ricci soliton satisfying suitable pinching conditions were obtained, for instance, by Cao and Hung \cite{CH}. In addition, Catino \cite{Catino} proved that every four-dimensional complete gradient shrinking Ricci soliton with nonnegative Ricci curvature and satisfying 
\begin{equation}
\label{lk10}
|W|R\leq \sqrt{3}\left(|\mathring{Ric}|-\frac{1}{2\sqrt{3}}R\right)^2
\end{equation} is a finite quotient of $\Bbb{R}^4,$ $\Bbb{S}^{3}\times\Bbb{R},$ or $\Bbb{S}^{4}.$ Moreover, as it was observed in \cite{Wu}, the nonnegative Ricci curvature $Ric\geq 0$ assumption in Catino's result can be removed. Later, by a different approach, Zhang \cite{Zhang2} showed that a four-dimensional complete gradient shrinking Ricci soliton with bounded and nonnegative Ricci curvature $0\leq Ric\leq C$ and satisfying 
\begin{equation}
\label{lk2}
|W|\leq \gamma \Big| |\mathring{Ric}|-\frac{1}{2\sqrt{3}}R\Big|,
\end{equation} for some constant $\gamma<1+\sqrt{3},$ is either flat or has $2$-positive Ricci curvature. 
Notice, however, that the pinching conditions (\ref{lk10}) and (\ref{lk2}) do not recover the complete gradient product Ricci soliton $\Bbb{S}^{2}\times\Bbb{R}^{2},$ which can also be considered as a gradient shrinking K\"ahler-Ricci soliton.

Now, we are ready to state our first result.

\begin{theorem}
\label{thmA}
Let $(M^4,\,g,\,f)$ be a four-dimensional complete  gradient shrinking Ricci soliton such that either $$|W^{+}|^{2} -\sqrt{6}|W^{+}|^{3}\geq \frac{1}{2}\langle (\mathring{Ric}\odot \mathring{Ric})^{+},W^{+}\rangle, $$ or
$$|W^{-}|^{2} -\sqrt{6}|W^{-}|^{3}\geq \frac{1}{2}\langle (\mathring{Ric}\odot \mathring{Ric})^{-},W^{-}\rangle.$$ 
Then $(M^4,\,g,\,f)$ is either 

\medskip
\begin{enumerate}
\item[(i)] Einstein, or
\smallskip

\item[(ii)] a finite quotient of either the Gaussian shrinking soliton $\Bbb{R}^4,$ or  $\Bbb{S}^{3}\times\Bbb{R}$, or  $\Bbb{S}^{2}\times\Bbb{R}^{2}.$
\end{enumerate}

\end{theorem}

\begin{remark} One new feature is that no Ricci curvature bound is assumed. Moreover, equality in Theorem \ref{thmA} for $W^{+}$ holds for the product Ricci soliton $\Bbb{S}^{2}\times\Bbb{R}^{2},$ regarded as a gradient K\"ahler-Ricci soliton with respect to the natural orientation induced from the complex structure on  $\Bbb{CP}^{1}\times\Bbb{C}$. Therefore, to some extent, Theorem \ref{thmA} can be considered as a gap theorem for $\Bbb{S}^{2}\times\Bbb{R}^{2}$. 
\end{remark}

\begin{remark}
\label{remarkE}
It is also interesting to estimate the value of $\langle (\mathring{Ric}\odot \mathring{Ric})^{+},W^{+}\rangle$ in order to obtain a condition depending on the norm of $\mathring{Ric}$ instead of a Kulkarni-Nomizu product $\odot$ in the right hand side of our assumption in Theorem \ref{thmA}. In fact, we can prove the inequality
\begin{equation}
\label{EqHu}
\langle (\mathring{Ric}\odot \mathring{Ric})^{+},W^{+}\rangle \leq \sqrt{6}|\mathring{Ric}|^{2}|W^{+}|; 
\end{equation} see Proposition \ref{propH} in Section \ref{Background} for more details. However, notice that the equality case of (\ref{EqHu}) does not hold for   $\Bbb{S}^{2}\times\Bbb{R}^{2}.$ Indeed, as shown in Example \ref{ex1a} in Section \ref{Background}, we have $$\langle (\mathring{Ric}\odot \mathring{Ric})^{+},W^{+}\rangle =\frac{\sqrt{6}}{3} |\mathring{Ric}|^{2}|W^{+}|$$ for $\Bbb{S}^{2}\times\Bbb{R}^{2}$. 
\end{remark}

Curvature estimates play a crucial role in the study of complete noncompact gradient shrinking Ricci soliton. In this context, it is well known that the curvature of three-dimensional gradient shrinking Ricci solitons must be nonnegative \cite{Ivey, Hamilton2, Chen}, before the complete classification was known. However, for dimension four or higher, this is no longer true according to the example of Feldman, Ilmanen and Knopf \cite{FIK}. Recently, there has been much progress concerning the curvature estimate of four-dimensional gradient Ricci solitons, led by the work of Munteanu and Wang \cite{MW} in which they proved that a complete four-dimensional gradient shrinking Ricci soliton must have bounded Riemann curvature tensor $Rm$, provided that its scalar curvature $R$ is bounded above by a constant. Similar curvature estimates were also obtained for complete four-dimensional gradient steady Ricci solitons by Cao and Cui \cite{CC}; see also Chan \cite{Chan}.

In our next result\footnote{This result, namely Theorem 2, was initially needed in the proof of Theorem 1. Later we were able to prove Theorem 1 without having to use it.}, inspired by the work in \cite{MW}, we shall prove curvature estimates for four-dimensional complete gradient shrinking Ricci solitons by assuming that its scalar curvature is suitable bounded by the potential function (hence allowing controlled quadratic growth). More precisely, we have established the following result which extends the curvature estimates by Munteanu and Wang.

\begin{theorem}
\label{thmRIC}
Let $(M^{4},\,g,\,f)$ be a four-dimensional complete noncompact gradient shrinking Ricci soliton satisfying $$R\leq A+\varepsilon f,$$ for some  constants $A>0$ and $0\leq\varepsilon<1.$ Then there exist positive constants $c_{0},$ $c_{1}$ and $c_{2}$ (independent of $f$) such that the following curvature estimates hold on $M$:

\medskip
\begin{enumerate}
\item $|Ric|\leq c_{0}+ (c_{1}\varepsilon)\, f;$
\vspace{0.3cm}
\item $|Rm|\leq c_{0}+ (c_{2}\varepsilon)\, f^{2}.$
\end{enumerate} 
\end{theorem}

\begin{remark}
We note that $R \le f$ for any gradient shrinking Ricci soliton (cf. Lemma \ref{lem1}). Recently, Chow, Freedman, Shin and Zhang \cite{CFSZ} showed that any four-di\-men\-sional gradient shrinking Ricci soliton arising as a blow-up limit of a finite time singularity of the Ricci flow must have curvature tensor satisfying $|Rm|\leq c (r+1)^{2},$ i.e., $Rm$ has at most quadratic growth.  It would be interesting to see if our estimate  (2) can be improved to $|Rm|\leq c_{0}+ (c_{2}\varepsilon)\, f$ so that $Rm$ also has at most quadratic growth in our situation. 
\end{remark}

\begin{remark} 
We point out that, by using the estimate on the Riemann curvature tensor obtained in Theorem \ref{thmRIC} and essentially the same arguments as in \cite{MW}, one can further derive the estimate 
 $$|\nabla Rm|\leq c_{3}+(c_{4}\varepsilon)\,f^{4}, $$ for some positive constants $c_{3}$ and $c_{4}.$
\end{remark}

\vspace{0.30cm}
The article is organized as follows. In Section \ref{Background}, we review some classical tensors and basic facts about gradient shrinking Ricci solitons. Moreover, we provide some key lemmas that will be used in the proof of the main results. In Section \ref{rigidityR}, we prove Theorem \ref{thmA}.
In Section \ref{Scurvatureestimate}, we present the proof of the curvature estimates stated in Theorem \ref{thmRIC}.

\section{Background}
\label{Background}

In this section we review some basic facts and present some lemmas that will be useful in the proof of the main results.

We start by recalling some special tensors in the study of curvature for a Riemannian manifold $(M^n,\,g)$ of dimension $n\ge 4.$  The first one is the Weyl tensor $W$ which is defined by the following decomposition formula
\begin{eqnarray}
\label{weyl}
R_{ijkl}&=&W_{ijkl}+\frac{1}{n-2}\big(R_{ik}g_{jl}+R_{jl}g_{ik}-R_{il}g_{jk}-R_{jk}g_{il}\big) \nonumber\\
 &&-\frac{R}{(n-1)(n-2)}\big(g_{jl}g_{ik}-g_{il}g_{jk}\big),
\end{eqnarray} where $R_{ijkl}$ stands for the Riemann curvature tensor. The second one is the Cotton tensor $C$ given by
\begin{equation}
\label{cotton} \displaystyle{C_{ijk}=\nabla_{i}R_{jk}-\nabla_{j}R_{ik}-\frac{1}{2(n-1)}\big(\nabla_{i}R
g_{jk}-\nabla_{j}R g_{ik}).}
\end{equation} Next, the Schouten tensor $A$ is defined by
\begin{equation}
\label{schouten} A_{ij}=R_{ij}-\frac{R}{2(n-1)}g_{ij}.
\end{equation} Combining Eqs. (\ref{weyl}) and (\ref{schouten}) we obtain
\begin{equation}
\label{WS} R_{ijkl}=\frac{1}{n-2}(A\odot g)_{ijkl}+W_{ijkl},
\end{equation} where $\odot$ is the Kulkarni-Nomizu product.

Recall that the Kulkarni-Nomizu product  $\odot,$ which takes two symmetric $(0,2)$-tensors and yields a $(0,4)$-tensor with the same algebraic symmetries as the curvature tensor, is defined by
\begin{eqnarray}
\label{eq76}
(\alpha \odot \beta)_{ijkl}=\alpha_{ik}\beta_{jl}+\alpha_{jl}\beta_{ik}-\alpha_{il}\beta_{jk}-\alpha_{jk}\beta_{il}.
\end{eqnarray}

As mentioned in the introduction, four-dimensional manifolds display various special features. For instance, the bundle of $2$-forms $\Lambda^2$ on a four-dimensional oriented Riemannian manifold can be invariantly decomposed as a direct sum, 
\begin{equation}
\label{lk1}
\Lambda^2=\Lambda^{+}\oplus\Lambda^{-}.
\end{equation} This decomposition is conformally invariant. In particular, let $\{e_{i}\}_{i=1}^{4}$ be an oriented orthonormal basis of the tangent space at any fixed point $p\in M^4,$ then it gives rise to bases of $\Lambda^{\pm}$ given by
\begin{equation}
\label{lkm1}
\Big\{e^{1}\wedge e^{2}\pm e^{3}\wedge e^{4},\,e^{1}\wedge e^{3}\pm e^{4}\wedge e^{2},\,e^{1}\wedge e^{4}\pm e^{2}\wedge e^{3}\Big\},
\end{equation} where each bi-vector has length $\sqrt{2}.$ Moreover, the decomposition (\ref{lk1}) allows us to conclude that the Weyl tensor $W$ is an endomorphism of $\Lambda^2=\Lambda^{+} \oplus \Lambda^{-} $ such that 
\begin{equation}
\label{ewq}
W = W^+\oplus W^-,
\end{equation} where $W^\pm:\Lambda^\pm M\longrightarrow\Lambda^\pm M$ are called the {\it self-dual} part and {\it anti-self-dual} part of the Weyl tensor $W$, respectively. Thereby, we may fix a point $p\in M^4$ and diagonalize $W^\pm$ such that $w_i^\pm,$ $1\le i \le 3,$ are their respective eigenvalues. Also, these eigenvalues satisfy
\begin{equation}
\label{eigenvalues} w_1^{\pm}\leq w_2^{\pm}\leq w_3^{\pm}\,\,\,\,\hbox{and}\,\,\,\,w_1^{\pm}+w_2^{\pm}+w_3^{\pm}
= 0.
\end{equation}

We now establish the following algebraic proposition.

\begin{proposition}
\label{proAlg}
Let $(M^{4},\,g)$ be a four-dimensional Riemannian manifold. Then we have:
\begin{enumerate}
\item $\frac{3}{2}(w_{3}^{+})^{2}\leq |W^{+}|^{2},$  and 
\item $\det W^{+}\leq \frac{1}{6}(w_{3}^{+})|W^{+}|^{2}\leq \frac{\sqrt{6}}{18}|W^{+}|^{3}.$ 
\end{enumerate} 
Moreover, equality holds in either case if and only if $w_{1}^{+}=w_{2}^{+}.$
\end{proposition}
\begin{proof}

Firstly, it follows from (\ref{eigenvalues}) that
\begin{eqnarray}
|W^{+}|^{2}&=&(w_{1}^{+})^{2}+(w_{2}^{+})^{2}+(w_{3}^{+})^{2}\nonumber\\&=& \frac{1}{2}\left(w_{1}^{+}+w_{2}^{+}\right)^{2}+\frac{1}{2}(w_{1}^{+}-w_{2}^{+})^{2}+(w_{3}^{+})^{2}\nonumber\\&\geq &\frac{1}{2}(w_{3}^{+})^{2}+(w_{3}^{+})^{2}=\frac{3}{2}(w_{3}^{+})^{2}.
\end{eqnarray} Moreover, the equality holds if and only if $w_{1}^{+}=w_{2}^{+},$ as asserted.   

Proceeding, we compute
\begin{eqnarray}
\label{eq12ws}
(w_{3}^{+})|W^{+}|^{2}-6\det W^{+}& =& 2w_{3}^{+}\left[(w_{3}^{+})^{2}+(w_{2}^{+})^{2}+ (w_{3}^{+})(w_{2}^{+})\right] -6 w_{1}^{+} w_{2}^{+}w_{3}^{+}\nonumber\\&=& 2(w_{3}^{+})^{3}-8w_{1}^{+} w_{2}^{+}w_{3}^{+}\nonumber\\&=& 2(w_{3}^{+})(w_{1}^{+}-w_{2}^{+})^{2}\ge 0,
\end{eqnarray} where we have used repeatedly that $w_1^{+}+w_2^{+}+w_3^{+}=0.$ The second inequality now follows from the above and inequality (1). Moreover, the equality case is straightforward. This finishes the proof of Proposition \ref{proAlg}. 

\end{proof}

As is well known, the curvature operator $Rm$ of a four-dimensional Riemannian manifold $M^4$ admits the following decomposition 

\begin{equation}\label{eq212}
Rm=
\left(
  \begin{array}{c|c}
    \\
W^{+} +\frac{R}{12}Id & \mathring{Ric} \\ [0.4cm]\hline\\

    \mathring{Ric}^{\star} & W^{-}+\frac{R}{12}Id  \\[0.4cm]
  \end{array}
\right),
\end{equation} where $\mathring{Ric}:\Lambda^{-}\to \Lambda^{+}$ stands for the traceless part of the Ricci curvature of $M^4.$

In the sequel, we present an estimate for $(\mathring{Ric}\odot \mathring{Ric})^{+}$ on oriented four-dimensional  Riemannian manifolds.

\begin{proposition}
\label{propH}
Let $(M^{4},\,g)$ be an oriented four-dimensional  Riemannian manifold. Then we have:
\begin{enumerate}
\item $|(\mathring{Ric}\odot \mathring{Ric})|^2 \leq 6|\mathring{Ric}|^{4}.$ Moreover,  equality holds if and only if $4|\mathring{Ric}^{2}|^{2}=|\mathring{Ric}|^{4},$ where $(\mathring{Ric}^{2})_{ik}=(\mathring{Ric})_{ip} (\mathring{Ric})_{kp}.$
\item $|(\mathring{Ric}\odot \mathring{Ric})^{+}|^2\leq 6|\mathring{Ric}|^{4}.$ Furthermore, equality holds if and only if $|(\mathring{Ric}\odot \mathring{Ric})^{-}|=0.$
\end{enumerate}
\end{proposition}

\begin{proof} First of all, the Kulkarni-Nomizu product yields $$(\mathring{Ric}\odot \mathring{Ric})_{ijkl}=2(\mathring{R}_{ik}\mathring{R}_{jl}-\mathring{R}_{il}\mathring{R}_{jk}).$$ 
Suppose $\mathring R_{ij}=\lambda_i\delta_{ij}$, then 
for each $i\ne j$,
$$(\mathring{Ric}\odot \mathring{Ric})_{ijkl}=2\lambda_i\lambda_j(\delta_{ik}\delta_{jl}-\delta_{il}\delta_{jk}).
$$
Therefore, we obtain
\[
\begin{split}|(\mathring{Ric}\odot \mathring{Ric})|^2&=8\sum_{i\ne j}\lambda_i^2\lambda_j^2=8\sum_{i=1}^4\left(\sum_{j\ne i} \lambda_i^2\lambda_j^2\right)\\
&=8\sum_{i=1}^4\lambda_i^2\left(\sum_{j\ne i} \lambda_j^2\right)
=8\sum_{i=1}^4\lambda_i^2\left(\sum_{j= 1}^4\lambda_j^2-\lambda_i^2\right)\\
&=8\left[\left(\sum_{i= 1}^4 \lambda_i^2\right)^2-\sum_{i=1}^4\lambda_i^4\right]
\\
&\le 8\left[\left(\sum_{i= 1}^4 \lambda_i^2\right)^2-\frac14\left(\sum_{i= 1}^4 \lambda_i^2\right)^2\right]=6|\mathring{Ric}|^{4}.
\end{split}
\]
In particular, equality holds if and only if $4|\mathring{Ric}^{2}|^{2}=|\mathring{Ric}|^{4}.$

Next, taking into account the decomposition (\ref{lk1}), one sees that $$|(\mathring{Ric}\odot \mathring{Ric})^{+}|^{2}\leq |\mathring{Ric}\odot \mathring{Ric}|^{2}\leq 6|\mathring{Ric}|^{4},$$ and consequently, 

$$|(\mathring{Ric}\odot \mathring{Ric})^{+}|\leq \sqrt{6}|\mathring{Ric}|^{2},$$ which gives the desired result.

\end{proof}

We highlight that the estimate (\ref{EqHu}) in Remark \ref{remarkE} is an immediate consequence of the Cauchy-Schwarz inequality combined with Proposition \ref{propH}.

For our purpose, as explained in the introduction, it is useful to examine the gradient shrinking Ricci soliton on $\Bbb{S}^2 \times \Bbb{R}^2.$ 

\begin{example}
\label{ex1a}
Let $\Big(\Bbb{S}^2 (\sqrt{2})\times \Bbb{R}^2,\,g,\,f\Big)$ denote the gradient shrinking Ricci soliton on the round cylinder $\Bbb{S}^2 (\sqrt{2})\times \Bbb{R}^2$, with the product metric $g$ and the potential function given by $f(x,y)=\frac{|y|^{2}}{4}+1,$ for $(x,y)\in\Bbb{S}^2 (\sqrt{2})\times \Bbb{R}^2.$ It is easy to see that $\Bbb{S}^2 (\sqrt{2})\times \Bbb{R}^2$ has constant scalar curvature $R=1$. In particular, we may choose a basis $\{ e_1, e_2, e_3, e_4\}$ such that $\{ e_1, e_2\}$ are tangent to the first factor and $\{ e_3, e_4\}$ tangent to the second factor. Therefore, under this coordinates, we have the expressions of Ricci curvature and traceless Ricci curvature given by
\begin{equation}
   Ric=
  \left[ {\begin{array}{cccc}
    \frac{1}{2} & 0 & 0 & 0 \\
   0 & \frac{1}{2} & 0 & 0 \\
   0 & 0 & 0 & 0 \\
      0 & 0 & 0 & 0 \\
  \end{array} } \right]
\end{equation} 
and 
\begin{equation}
\label{eqMR}
   \mathring{Ric}=
  \left[ {\begin{array}{cccc}
    \frac{1}{4} & 0 & 0 & 0 \\
   0 & \frac{1}{4} & 0 & 0 \\
   0 & 0 & -\frac{1}{4} & 0 \\
      0 & 0 & 0 & -\frac{1}{4} \\
  \end{array} } \right].
\end{equation} Next, by using the decomposition $\Lambda^2=\Lambda^{+}\oplus\Lambda^{-},$ one sees that
\[
Rm= \left[ {\begin{array}{cccccc}
    \frac{1}{4} & 0 & 0 & \frac{1}{4} & 0& 0  \\
   0 & 0 & 0 & 0 & 0& 0 \\
   0 & 0 & 0 & 0 & 0& 0 \\
  \frac{1}{4} & 0 & 0 & \frac{1}{4} & 0& 0 \\
  0 & 0 & 0 & 0 & 0& 0 \\
   0 & 0 & 0 & 0 & 0& 0 
  \end{array} } \right].
\]
 Moreover, taking into account (\ref{ewq}) and (\ref{eq212}), we infer that $W^{\pm}$ has precisely two distinct eigenvalues, i.e.,
\begin{equation*}
   W^{\pm}=
  \left[ {\begin{array}{ccc}
     \frac{1}{6} & 0 & 0  \\
   0 & - \frac{1}{12} & 0 \\
   0 & 0 & - \frac{1}{12} \\
  \end{array} } \right].
\end{equation*}

Now, we claim that $\Bbb{S}^2(\sqrt{2})  \times \Bbb{R}^2$ satisfies 
\begin{equation}
\label{EqRR}
\langle (\mathring{Ric}\odot \mathring{Ric})^{+},W^{+}\rangle=\frac{1}{24}.
\end{equation} Indeed, a direct computation using the Kulkarni-Nomizu product (\ref{lk1}) gives $$(\mathring{Ric}\odot \mathring{Ric})_{ijkl}=2(\mathring{R}_{ik}\mathring{R}_{jl}-\mathring{R}_{il}\mathring{R}_{jk}).$$  In particular, we have $|\mathring{Ric}\odot \mathring{Ric}|^2=\frac{3}{8}.$ Next, given a $2$-form $\omega,$ we have
\begin{eqnarray}
\Big((\mathring{Ric}\odot \mathring{Ric})\omega\Big)_{ij}&=& \frac{1}{2}(\mathring{Ric}\odot \mathring{Ric})_{ijkl}\omega_{kl}\nonumber\\&=&(\mathring{R}_{ik}\mathring{R}_{jl}-\mathring{R}_{il}\mathring{R}_{jk})\omega_{kl}\nonumber\\&=&\mathring{R}_{ik}\mathring{R}_{jl}\omega_{kl}-\mathring{R}_{il}\mathring{R}_{jk}\omega_{kl}.
\end{eqnarray} 
By using that $\mathring{Ric}=\lambda_{i}g_{ij}$ we achieve
\begin{eqnarray}
\Big((\mathring{Ric}\odot \mathring{Ric})\omega\Big)_{ij}&=&\lambda_{i}\delta_{ik}\lambda_{j}\delta_{jl}\omega_{kl}-\lambda_{i}\delta_{il}\lambda_{j}\delta_{jk}\omega_{kl}\nonumber\\&=& \lambda_{i}\lambda_{j}\omega_{ij}-\lambda_{i}\lambda_{j}\omega_{ji}\nonumber\\&=& 2\lambda_{i}\lambda_{j}\omega_{ij}.
\end{eqnarray} 
To proceed, we remember that the components of $e^{p}\wedge e^{q}$ are given by $\delta_{ij}^{pq}=(e^{p}\wedge e^{q})_{ij},$ where $\delta_{ij}^{pq}$ is the generalized Kronecker delta symbol, which is defined to be $+1$ if $(p,q)=(i,j),$ $-1$ if $(p,q)=(j,i),$ and $0$ otherwise. Therefore, taking into account the basis (\ref{lkm1}), for $\omega_{1}^{+}=e^{1}\wedge e^{2}+ e^{3}\wedge e^{4},$ we deduce that
\begin{eqnarray*}
(\mathring{Ric}\odot \mathring{Ric})(\omega_{1}^{+})&=& \frac{1}{2}(\mathring{Ric}\odot \mathring{Ric})_{ijkl}(\omega_{1}^{+})_{kl}\nonumber\\&=& 2\lambda_{i}\lambda_{j}(\delta_{12}^{ij}+\delta_{34}^{ij})\nonumber\\&=&2\lambda_{1}\lambda_{2}\delta_{12}^{ij}+2\lambda_{3}\lambda_{4}\delta_{34}^{ij}.
\end{eqnarray*} 
Then, it follows from (\ref{eqMR}) that
\begin{eqnarray}
(\mathring{Ric}\odot \mathring{Ric})(\omega_{1}^{+})&=& 2\frac{1}{16}\delta_{12}^{ij}+2\frac{1}{16}\delta_{34}^{ij}\nonumber\\&=&\frac{1}{8}\omega_{1}^{+}.
\end{eqnarray} Similarly, for $\omega_{2}^{+}=e^{1}\wedge e^{3}+ e^{4}\wedge e^{2}$ and $\omega_{3}^{+}=e^{1}\wedge e^{4}+ e^{2}\wedge e^{3},$ one concludes that
$$(\mathring{Ric}\odot \mathring{Ric})(\omega_{2}^{+})=-\frac{1}{8}\omega_{2}^{+}$$ and $$(\mathring{Ric}\odot \mathring{Ric})(\omega_{3}^{+})=-\frac{1}{8}\omega_{3}^{+}.$$ Recalling that $(\mathring{Ric}\odot \mathring{Ric})^{\pm}(\omega)=(\mathring{Ric}\odot \mathring{Ric})(\omega^{\pm}),$ we obtain
\begin{equation}
(\mathring{Ric}\odot \mathring{Ric})^{+}=
  \left[ {\begin{array}{ccc}
    \frac{1}{8} & 0 & 0 \\
   0 & -\frac{1}{8} & 0 \\
   0 & 0 & -\frac{1}{8} \\
  \end{array} } \right].
\end{equation} Consequently, $$\langle (\mathring{Ric}\odot \mathring{Ric})^{+},W^{+}\rangle=\frac{1}{48}+\frac{1}{96}+\frac{1}{96}=\frac{1}{24},$$ which proves our claim.  

\end{example}

In the sequel, we recall some important features of gradient shrinking Ricci solitons (cf. \cite{Hamilton2}).

\begin{lemma}
\label{lem1}
Let $\big(M^4,\,g,\,f\big)$ be a four-dimensional gradient shrinking Ricci soliton. Then we have:
\begin{enumerate}
\item $R+\Delta f=2.$
\item $\frac{1}{2}\nabla R=Ric(\nabla f).$
\item $\Delta_{f} R= R-2|Ric|^{2}.$
\item $R+|\nabla f|^{2}= f$ (after normalizing).
\item $\Delta_{f} R_{ij}=R_{ij}-2R_{ikjl}R_{kl}.$
\item $\Delta_{f} Rm =Rm+ Rm\ast Rm.$
\item $\nabla_{l}R_{ijkl}=\nabla_{j}R_{ik}-\nabla_{i}R_{jk}=R_{ijkl}f_{l}.$
\end{enumerate} Here, $\Delta_f\cdot:=\Delta\cdot-\nabla_{\nabla f}\cdot$ stands for the drifted Laplacian. 
\end{lemma}

In \cite{Chen}, Chen showed that every complete gradient shrinking Ricci soliton has nonnegative scalar curvature  $R\ge 0.$ Concerning the potential function $f,$ Cao and Zhou \cite{CZ} have proved that 
\begin{equation}
\label{eqfbeh}
\frac{1}{4}\Big(r(x)-c\Big)^{2}\leq f(x)\leq \frac{1}{4}\Big(r(x)+c\Big)^{2},
\end{equation} for all $r(x)\geq r_{0}.$ Moreover, they showed that every complete noncompact gradient shrinking Ricci soliton has at most Euclidean volume growth (see \cite[Theorem 1.2]{CZ}). These asymptotic estimates are optimal in the sense that they are achieved by the Gaussian shrinking soliton.

In \cite{derd1}, Derdzi\'nski showed that every oriented four-dimensional Einstein manifold $(M^{4},\,g)$ satisfies the Weitzenb\"ock formula 
$$\Delta |W^{\pm}|^{2}=2|\nabla W^{\pm}|^{2}+R|W^{\pm}|^{2}-36 \det W^{\pm}.$$ The Weitzenb\"ock formula is a powerful ingredient in the theory of canonical metrics on four-dimensional manifolds. It may be used to obtain classification results as well as rule out some potential new examples.

Now we recall a useful Weitzenb\"ock type formula for gradient Ricci solitons \cite{CH, Wu1}. This formula will play a key role in the proof of Theorem \ref{thmA}.

\begin{proposition}
\label{propB}
Let $(M^4,\,g,\,f)$ be a four-dimensional gradient shrinking Ricci soliton. Then we have:
$$\Delta_{f} |W^{\pm}|^{2}=2|\nabla W^{\pm}|^{2}+2|W^{\pm}|^{2}-36 \det W^{\pm}-\langle (\mathring{Ric}\odot \mathring{Ric})^{\pm},W^{\pm}\rangle,$$
where $\odot$ stands for the Kulkarni-Nomizu product and $\Delta_f\cdot=\Delta\cdot-\nabla_{\nabla f}\cdot$ is the drifted Laplacian. 
\end{proposition}

\section{Proof of Theorem \ref{thmA}}
\label{rigidityR}

In this section, we will present the proof of Theorem \ref{thmA}. For convenience, we restate it here.

\begin{theorem}
\label{thmA1}
Let $(M^4,\,g,\,f)$ be a four-dimensional complete  gradient shrinking Ricci soliton such that either $$|W^{+}|^{2} -\sqrt{6}|W^{+}|^{3}\geq \frac{1}{2}\langle (\mathring{Ric}\odot \mathring{Ric})^{+},W^{+}\rangle, $$ or
$$|W^{-}|^{2} -\sqrt{6}|W^{-}|^{3}\geq \frac{1}{2}\langle (\mathring{Ric}\odot \mathring{Ric})^{-},W^{-}\rangle.$$ 
 Then $(M^4,\,g,\,f)$ is either 

\smallskip
\begin{enumerate}
\item[(i)] Einstein, or
\smallskip
\item[(ii)] a finite quotient of either the Gaussian shrinking soliton $\Bbb{R}^4,$ or  $\Bbb{S}^{3}\times\Bbb{R}$, or  $\Bbb{S}^{2}\times\Bbb{R}^{2}.$
\end{enumerate}

\end{theorem}

\begin{proof}
To begin with, we invoke Proposition \ref{propB} to obtain
\begin{equation}
\label{eqCH}
\Delta_{f} |W^{\pm}|^{2}=2|\nabla W^{\pm}|^{2}+2|W^{\pm}|^{2}-36 \det W^{\pm}-\langle (\mathring{Ric}\odot \mathring{Ric})^{\pm},W^{\pm}\rangle.
\end{equation}  Since $\Delta_{f}|W^{\pm}|^{2}=2|W^{\pm}|\Delta_{f}|W^{\pm}|+2|\nabla |W^{\pm}||^{2},$ we deduce that

\begin{eqnarray*}
2|W^{\pm}|\Delta_{f}|W^{\pm}| &=& 2|\nabla W^{\pm}|^{2}-2|\nabla |W^{\pm}||^{2}+2|W^{\pm}|^{2} -36 \det W^{\pm}\nonumber\\&& -\langle (\mathring{Ric}\odot \mathring{Ric})^{\pm},W^{\pm}\rangle.
\end{eqnarray*} Now, we use the Kato inequality to infer

\begin{eqnarray}
\label{eqs2}
2|W^{\pm}|\Delta_{f}|W^{\pm}| &\geq & 2|W^{\pm}|^{2} -36 \det W^{\pm} -\langle (\mathring{Ric}\odot \mathring{Ric})^{\pm},\,W^{\pm}\rangle.
\end{eqnarray}
Thereby, by Proposition \ref{proAlg} (2) we achieve

\begin{equation}
\label{eqq1p}
|W^{\pm}|\Delta_{f}|W^{\pm}| \geq  |W^{\pm}|^{2} -\sqrt{6}|W^{\pm}|^{3} -\frac{1}{2}\langle (\mathring{Ric}\odot \mathring{Ric})^{\pm},W^{\pm}\rangle.
\end{equation}

By assumption, we have 
\begin{equation} 
\label{ekpi}
|W^{+}|^{2} -\sqrt{6}|W^{+}|^{3}\geq \frac{1}{2}\langle (\mathring{Ric}\odot \mathring{Ric})^{+},W^{+}\rangle.
\end{equation} 
In this case, it follows from (\ref{eqq1p}) that $|W^{\pm}|\Delta_{f}|W^{+}|$ does not change sign. To proceed, we first need to prove that $|W^{+}|\in L^{2}(e^{-f}dV_{g}).$ Indeed, it follows from (\ref{ekpi}) jointly with the Cauchy-Schwarz inequality and Proposition \ref{propH} (1) that

\begin{eqnarray*}
\sqrt{6}|W^{+}|^{3}&\leq & |W^{+}|^{2}-\frac{1}{2}\langle (\mathring{Ric}\odot \mathring{Ric})^{+},W^{+}\rangle\nonumber\\ &\leq & |W^{+}|^{2} +\frac{\sqrt{6}}{2}|\mathring{Ric}|^{2}|W^{+}|\nonumber\\ &\leq & |W^{+}|^{2}+\frac{\sqrt{6}}{2}|Ric|^{2}|W^{+}|,
\end{eqnarray*}
 so that

\begin{equation}
\sqrt{6}|W^{+}|^{2}\leq |W^{+}|+\frac{\sqrt{6}}{2}|Ric|^{2}\leq \frac{\sqrt{6}}{2}|W^{+}|^2+\frac1{2\sqrt{6}}+\frac{\sqrt{6}}{2}|Ric|^{2}.
\end{equation} 
Upon integrating this above expression over $M^4,$ one sees that

\begin{eqnarray}
\label{eqqw1}
\int_{M}|W^{+}|^{2}e^{-f}dV_{g} &\leq  &\frac{1}{6}\int_{M}e^{-f}dV_{g}+ \int_{M}|Ric|^{2}e^{-f}dV_{g}.
\end{eqnarray} 
Next, it is well-known (see, e.g., Corollary 1.1 in \cite{CZ}) that the weighted volume of $M^4$ is finite, i.e., $$\int_{M}e^{-f}dV_{g}<\infty.$$ Moreover, Munteanu and Sesum \cite[Theorem 1.1]{MS} proved that

$$\int_{M}|Ric|^{2}e^{-f}dV_{g}<\infty.$$ Therefore, it follows from (\ref{eqqw1}) that $|W^{+}|$ is $L^{2}_{f}$-integrable. 

From (\ref{eqq1p}) and the assumption in Theorem \ref{thmA1}, we have $$|W^{+}|\Delta_{f}|W^{+}|\geq 0.$$ Let  $\varphi:M \to \Bbb{R}$ be a smooth cut-off function such that $\varphi=1$ on $B_{p}(r)$ (a geodesic ball centered at a fixed point $p\in M$ of radius $r$), $\varphi=0$ outside of $B_{p}(2r)$ and $|\nabla \varphi|\leq \frac{c}{r},$ where $c$ is a constant. Therefore, we infer 

\begin{eqnarray}
0&\geq &-\int_{M}\varphi^{2}|W^{+}|\Delta_{f}|W^{+}| e^{-f}dV_{g}=\int_{M}\langle \nabla (\varphi^{2}|W^{+}|),\,\nabla |W^{+}|\rangle e^{-f}dV_{g}\nonumber\\&=& \int_{M}\varphi^{2}|\nabla |W^{+}||^{2}e^{-f}dV_{g}+2\int_{M}\varphi |W^{+}|\langle \nabla |W^{+}|,\nabla \varphi\rangle e^{-f}dV_{g}\nonumber\\ &=& \int_{M}|\varphi \nabla |W^{+}|+|W^{+}|\nabla \varphi|^{2} e^{-f}dV_{g}-\int_{M}|W^{+}|^{2}|\nabla \varphi|^{2}e^{-f}dV_{g},
\end{eqnarray} consequently, $$\int_{M}|\nabla (\varphi |W^{+}|)|^{2}e^{-f}dV_{g} \leq \int_{M}|W^{+}|^{2}|\nabla \varphi|^{2} e^{-f}dV_{g}.$$ Thus, one obtains that 

\begin{eqnarray}
\int_{B(r)}|\nabla |W^{+}||^{2}e^{-f}dV_{g}&\leq & \int_{M}|\nabla (\varphi |W^{+}|)|^{2}e^{-f}dV_{g}\nonumber\\&\leq & \int_{B_{p}(2r)\backslash B_{p}(r)}|W^{+}|^{2}|\nabla \varphi|^{2} e^{-f}dV_{g}\leq \frac{c^{2}}{r^{2}}\int_{M}|W^{+}|^{2} e^{-f}dV_{g}.
\end{eqnarray} Then, since $|W^{+}|$ is $L^{2}_{f}$-integrable we conclude that the right hand side tends to zero as $r\to \infty.$ In conclusion, $|W^{+}|$ must be constant on $M^4.$

We now have two cases to be analyzed, namely, $|W^{+}|=0$ and $|W^{+}|\neq 0.$
Now,  if  $|W^{+}|=0,$ then we can use Theorem 1.2 in \cite{CW} to conclude that $M^4$ is either Einstein or a finite quotient of either the Gaussian shrinking soliton $\Bbb{R}^4$ or the round cylinder $\Bbb{S}^{3}\times\Bbb{R}.$

On the other hand, if $|W^{+}|$ is a nonzero constant, then we may use (\ref{eqCH}), assumption (\ref{ekpi}) and Proposition \ref{proAlg} to infer that $\nabla W^{+}=0$ and 
\begin{eqnarray*}
0 &=& 2|W^{+}|^{2}-36 \det W^{+}-\langle (\mathring{Ric}\odot \mathring{Ric})^{+},W^{+}\rangle\\
&=& 2|W^{+}|^{2} -2\sqrt{6}|W^{+}|^{3}  - \langle (\mathring{Ric}\odot \mathring{Ric})^{+},W^{+}\rangle.
 \end{eqnarray*} 
 Therefore, since now equality holds in Proposition \ref{proAlg} (2), it follows that $W^{+}$ has precisely two distinct eigenvalues. Moreover, it is obvious that $\nabla W^{+}=0$ implies $\delta W^{+}=0$ (i.e., half harmonic Weyl
curvature). Furthermore, by Proposition 5 in \cite{derd1}, $(M^4,\, g)$ is actually K\"ahler. In any case, we are in a position to apply Theorem 1.1 in \cite{Wu} to conclude that $M^4$ is either Einstein or a finite quotient of $\Bbb{S}^{2}\times\Bbb{R}^{2}.$

Obviously, if we change $W^{+}$ by $W^{-}$ in (\ref{ekpi}) the proof of the theorem is exactly the same. So, the proof is completed.

\end{proof}

\section{Proof of Theorem \ref{thmRIC}}
\label{Scurvatureestimate}

In this section, we will investigate new curvature estimates for a four-dimensional gradient shrinking Ricci soliton assuming that its scalar curvature $R$ is suitable bounded by the potential function $f,$ namely,  
\begin{equation}
\label{eq2w}
R\leq A+ \varepsilon f
\end{equation} for some constants $A>0$ and $0\leq \varepsilon<1.$ To this end, it is fundamental to recall a relevant estimate for the curvature operator $Rm$ of gradient shrinking Ricci solitons observed by Munteanu and Wang \cite{MW} (see also \cite[Lemma 1]{Chan}).

\begin{proposition}
\label{propA}
Let $(M^{4},\,g,\,f)$ be a four-dimensional complete noncompact gradient shrinking Ricci soliton. Then there exists a universal positive constant $C_{0}$ such that 

$$|Rm|\leq C_{0}\left(|Ric|+\frac{|\nabla Ric|}{|\nabla f|}\right)$$ 
whenever $|\nabla f|\ne 0$. 
\end{proposition}

Now, we are going to present the proof of the curvature estimates asserted in Theorem \ref{thmRIC} that will be stated again here.

\begin{theorem}
Let $(M^{4},\,g,\,f)$ be a four-dimensional complete noncompact gradient shrinking Ricci soliton satisfying $$R\leq A+\varepsilon f,$$ for some constants $A>0$ and $0\leq\varepsilon<1.$ Then there exist $c_{0},$ $c_{1}$ and $c_{2}$  such that the following curvature estimates hold on $M$:

\medskip
\begin{enumerate}
\item $|Ric|\leq c_{0}+ (c_{1}\varepsilon)\, f;$
\vspace{0.3cm}
\item $|Rm|\leq c_{0}+ (c_{2}\varepsilon)\, f^{2}.$
\end{enumerate} 
\end{theorem}

\begin{proof} Throughout the proof, for simplicity, we shall use $c$ or $C$ to denote universal constants which may vary from line to line.  

We start by proving the estimate on the Ricci tensor $Ric$. In the first part of the proof, we shall follow the arguments by Munteanu-Wang \cite{MW}. To begin with, notice that 
\begin{eqnarray}
\label{eq34a}
\Delta_{f}|Ric|^{2} &\geq & 2|\nabla Ric|^{2} -c|Rm||Ric|^{2}\nonumber\\&\geq & 2|\nabla Ric|^{2}-c|Ric|^{3}-c\frac{|\nabla Ric|}{|\nabla f|}|Ric|^{2},
\end{eqnarray} where we have used Lemma \ref{lem1} and Proposition \ref{propA}.

On the other hand, for any $0<a<1,$ one easily verifies from Lemma \ref{lem1} (3) that
\begin{eqnarray}
\Delta_{f}\Big(\frac{1}{R^{a}}\Big)&=&-aR^{-a-1}\Delta_{f}R+a(a+1)\frac{|\nabla R|^{2}}{R^{a+2}}\nonumber\\&=& -\frac{a}{R^{a}}+2a\frac{|Ric|^{2}}{R^{a+1}}+a(a+1)\frac{|\nabla R|^{2}}{R^{a+2}}.
\end{eqnarray} 
This jointly with (\ref{eq34a}) yields
\begin{eqnarray}
\label{eq85a}
\Delta_{f}\Big(\frac{|Ric|^{2}}{R^{a}}\Big) &=& \frac{\Delta_{f}|Ric|^{2}}{R^{a}}+|Ric|^{2}\Delta_{f}\Big(\frac{1}{R^{a}}\Big) + 2\Big\langle \nabla |Ric|^{2},\nabla \Big(\frac{1}{R^{a}}\Big)\Big\rangle\nonumber\\
&\geq &\frac{1}{R^{a}}\left[2|\nabla Ric|^{2}-c|Ric|^{3}-c\frac{|\nabla Ric|}{|\nabla f|}|Ric|^{2}\right] \nonumber\\&&+|Ric|^{2}\left[-\frac{a}{R^{a}}+2a\frac{|Ric|^{2}}{R^{a+1}}+a(a+1)\frac{|\nabla R|^{2}}{R^{a+2}}\right]\nonumber\\&&+2\Big\langle \nabla |Ric|^{2},\nabla \Big(\frac{1}{R^{a}}\Big)\Big\rangle.
\end{eqnarray} 

By the Cauchy-Schwarz and Kato's inequalities, we have 
\begin{eqnarray*}
2\Big\langle \nabla |Ric|^{2},\nabla \Big(\frac{1}{R^{a}}\Big)\Big\rangle &\geq & -a(a+1)\frac{|Ric|^{2}|\nabla R|^{2}}{R^{a+2}}-\frac{4a}{a+1}\frac{|\nabla Ric|^{2}}{R^{a}}.
\end{eqnarray*} 
Therefore, returning to (\ref{eq85a}), we achieve
\begin{eqnarray}
\Delta_{f}\Big(\frac{|Ric|^{2}}{R^{a}}\Big) &\geq &\frac{2(1-a)}{(1+a)}\frac{|\nabla Ric|^{2}}{R^{a}}-c\frac{|\nabla Ric|}{|\nabla f|}\frac{|Ric|^{2}}{R^{a}}-c\frac{|Ric|^{3}}{R^{a}}\nonumber\\&&-a\frac{|Ric|^{2}}{R^{a}}+2a\frac{|Ric|^{4}}{R^{a+1}}.
\end{eqnarray} 
Taking into account that 
\begin{eqnarray}
\frac{2(1-a)}{(1+a)}\frac{|\nabla Ric|^{2}}{R^{a}}-c\frac{|\nabla Ric|}{|\nabla f|}\frac{|Ric|^{2}}{R^{a}}\geq -\frac{(1+a)}{8(1-a)}\frac{c^{2}}{|\nabla f|^{2}}\frac{|Ric|^{4}}{R^{a}},\nonumber
\end{eqnarray} 
we then obtain
\begin{eqnarray}
\label{eq5tg}
\Delta_{f}\Big(\frac{|Ric|^{2}}{R^{a}}\Big) &\geq &\Big(2a-\frac{c}{(1-a)}\frac{R}{|\nabla f|^{2}}\Big)\frac{|Ric|^{4}}{R^{a+1}}-c\frac{|Ric|^{3}}{R^{a}}-a\frac{|Ric|^{2}}{R^{a}}.
\end{eqnarray}

At the same time, we already know from Lemma \ref{lem1} that $R+|\nabla f|^{2}=f$ and hence, our assumption implies $$|\nabla f|^{2}\geq (1-\varepsilon)f-A.$$ From this, it follows that 
$$ \frac{R}{|\nabla f|^2}\leq \frac{R}{(1-\varepsilon)f-A} \leq \frac{\varepsilon f+A}{(1-\varepsilon)f-A}.$$ Plugging this into (\ref{eq5tg}), one obtains that

\begin{equation}
\Delta_{f}\Big(\frac{|Ric|^{2}}{R^{a}}\Big) \geq \left(2a-\frac{c}{(1-a)}\frac{(\varepsilon f+A)}{[(1-\varepsilon)f-A]}\right)\frac{|Ric|^{4}}{R^{a+1}}-c\frac{|Ric|^{3}}{R^{a}}-a\frac{|Ric|^{2}}{R^{a}}.
\end{equation}   
Now, we need to analyze the term $$\Gamma=2a-\frac{c}{(1-a)}\frac{(\varepsilon f+A)}{[(1-\varepsilon)f-A]}$$ in the above expression. Indeed, considering $a=\frac{1}{2},$ one sees that $$\Gamma=1-\frac{2c(\varepsilon f +A)}{[(1-\varepsilon)f-A]}.$$ Observe that $$\frac{2c(\varepsilon f +A)}{[(1-\varepsilon)f-A]}\leq \frac{1}{2}\,\,\,\,\,\hbox{whenever}\,\,\,\,\,A\leq \frac{(1-\varepsilon-4c\varepsilon)}{(1+4c)}f.$$ Notice that we can assume our $\varepsilon$ has been chosen in such a way that $\varepsilon < 1/(1+4c)$. 
Therefore, the asymptotic behaviour of $f$ in (\ref{eqfbeh}) guarantees that there exists $r_{0}>0,$ depending only on $A$ and $\varepsilon,$ so that $\Gamma\geq \frac{1}{2}$ and $(1-\varepsilon)f-A\ge 1$ on $M\setminus D(r_{0}).$ Here, $D(r_{0})=\{f \leq r_0\}$.
In view of this, it then follows that

 \begin{eqnarray}
 \label{equ1}
\Delta_{f}u \geq  \frac 12 u^{2}R^{-\frac{1}{2}}-cu^{\frac{3}{2}}R^{\frac{1}{4}}-u
\end{eqnarray} on $M\setminus D(r_{0}),$ where $u:=\frac{|Ric|^{2}}{\sqrt {R}}.$

In the sequel, let $\eta:[0,+\infty)\to [0,\,1]$ be a smooth function such that 
\begin{equation}
\label{eta} {\eta (t) =\left\{
       \begin{array}{lll}
  0, \ \ \quad  t\in [0,\frac{1}{2}],\\[2mm]
    1, \ \ \quad  t\in [1,2],\\[2mm]
    0, \ \ \quad  t\in [3,+\infty).
       \end{array}
    \right.}
    \end{equation} 
    Moreover, we consider $\varphi(x)=\eta \Big(\frac{f(x)}{\rho^{2}}\Big),$ where $\rho> 2r_{0}$ is a constant.

 \begin{remark}  
 \label{remarkZ}
Notice, in particular, that if $\frac{f(x)}{\rho^{2}}\in [1,2],$ then $\varphi(x)=1.$ Furthermore, it is easy to see that if $\frac{f(x)}{\rho^{2}}> 3,$ then $\varphi(x)=0.$ So, one needs only consider the case $\frac{f(x)}{\rho^{2}}\leq 3.$ Similarly, we deduce that $\varphi(x)=0$ for all $x$ such that $\frac{f(x)}{\rho^{2}}<1/2.$ Hence, it suffices to analyze those points $x\in M$ such that $$\frac{\rho^{2}}{2}\leq f(x)\leq 3\rho^{2}.$$
\end{remark}

Proceeding, it follows that 
\begin{equation}
\label{eqq14}
|\nabla \varphi|\leq |\eta'|\frac{|\nabla f|}{\rho^{2}}\leq |\eta'|\frac{\sqrt{f}}{\rho^{2}}\leq |\eta'| \frac{\sqrt{3}}{\rho}.
\end{equation} Notice also that 

\begin{equation}
\label{eqq142}
\Delta_{f}\varphi=\eta''\frac{|\nabla f|^{2}}{\rho^{4}}+\eta'\frac{1}{\rho^{2}}\Delta_{f}f=\eta''\frac{|\nabla f|^{2}}{\rho^{4}}+\eta'\frac{1}{\rho^{2}}(2-f)\leq C,
\end{equation} for a universal constant $C>0.$ Here, we have used that $\Delta_{f}f=2-f,$ which is a direct consequence of equations (1) and (4) of Lemma \ref{lem1}.

From now on, we set $G=\varphi^{2}u=\varphi^{2}\frac{|Ric|^{2}}{\sqrt {R}}$. Then, one obtains that
    
    $$\nabla G=\nabla \Big(\varphi^{2}u\Big)=(\nabla u)\varphi^{2}+u\nabla \varphi^{2}$$ and $$\Delta_{f}\varphi^{2}=2|\nabla \varphi|^{2}+2\varphi\Delta_{f}\varphi.$$ These facts combined with (\ref{equ1}) allow us to infer
    \begin{eqnarray}
    \label{equ23}
   \varphi^{2}\Delta_{f}G &=& \varphi^{2} \Delta_{f} (u\varphi^{2})\nonumber\\&\geq &\varphi^{4}\Big[\frac{1}{2}u^{2}R^{-\frac{1}{2}}-cu^{\frac{3}{2}}R^{\frac{1}{4}}-cu\Big]+ \varphi^{2}u\Big[2|\nabla \varphi|^{2}+2\varphi\Delta_{f}\varphi\Big]\nonumber\\&&+2\langle \nabla G,\nabla \varphi^{2}\rangle -2u\langle \nabla \varphi^{2},\nabla \varphi^{2}\rangle\nonumber\\&\geq & \frac{1}{2}G^{2}R^{-\frac{1}{2}}-cG^{\frac{3}{2}}R^{\frac{1}{4}}-cG+2\langle \nabla G, \nabla \varphi^{2}\rangle,
\end{eqnarray} where we also used that $0\leq \varphi\leq 1.$ Besides, assuming that $G$ achieves its maximum at some point $q\in D(3\rho^{2})=\{x\in M\!: f(x)\leq 3\rho^{2}\}$, we have 

$$0\geq \frac{1}{2}G^{2}(q)R^{-\frac{1}{2}}(q)-cG^{\frac{3}{2}}(q)R^{\frac{1}{4}}(q)-cG(q).$$ Rearranging terms and considering $v=G^{\frac{1}{2}},$ we get 

$$0\geq \frac{1}{2}R^{-\frac{1}{2}}v^{2}-cR^{\frac{1}{4}}v-c,$$ at point $q.$ Thus, by computing the discriminant, one sees that $v(q)\leq cR^{\frac{3}{4}}(q).$ Consequently,  $$G(q)\leq c R^{\frac{3}{2}}(q).$$ Now, on $D(3\rho^{2}),$ we have $$G(x)\leq G(q)\leq c R^{\frac{3}{2}}(q)\leq c (\varepsilon f(q)+A)^{\frac{3}{2}}.$$ Therefore, on $D(2\rho^{2})\setminus D(\rho^{2}),$ we deduce that

\begin{eqnarray}
|Ric|^{2}(x)&=&G(x)R^{\frac{1}{2}}(x)\leq G(q)R^{\frac{1}{2}}(x)\nonumber\\ &\leq & c  (\varepsilon f(q)+A)^{\frac{3}{2}}  (\varepsilon f(x)+A)^{\frac{1}{2}} \leq C'  (\varepsilon f(x)+A)^{2},
\end{eqnarray} where we have used the fact that $G=0$ on $D({\rho^{2}}/2)$ and hence $q\in D(3\rho^{2})\setminus D(\rho^{2}/2)$, i.e., $f(q)\geq \frac{\rho^{2}}{2}.$ 

Finally, since $\rho$ can be arbitrarily large, we conclude that $$|Ric|\leq c_{0}+ (c_{1}\varepsilon)\, f$$  on $M\setminus D(r_{0})$.  We can finish the proof of the first assertion in the theorem by choosing $c_0 $ large enough so that estimate (1) in Theorem 4 holds on all $M$.

\vspace{0.40cm}

Next, we deal with the second assertion in the theorem, i.e., the estimate on the Riemann curvature tensor. Again, we adapt the argument by Munteanu-Wang \cite{MW}. Initially, by assertion (6) of Lemma \ref{lem1} and Kato's inequality, one observes that

\begin{equation}
\label{eqe4a}
\Delta_{f}|Rm|\geq -c |Rm|^{2}.
\end{equation} 

On the other hand, it follows from Proposition \ref{propA} that

\begin{eqnarray*}
\Delta_{f}|Ric|^{2}&\geq & 2|\nabla Ric|^{2}+2|Ric|^{2}-4|Rm||Ric|^{2}\nonumber\\
&\geq & 2|\nabla Ric|^{2}+2|Ric|^{2}-C|Ric|^{3}-C\frac{|\nabla Ric|}{|\nabla f|}|Ric|^{2}.
\end{eqnarray*}
This jointly with (\ref{eqe4a}) yields

\begin{eqnarray}
\label{ineqA12}
\Delta_{f}\Big(|Rm|+|Ric|^{2}\Big) &\geq & |Rm|^{2}-(c+1)|Rm|^{2}+2|\nabla Ric|^{2}+2|Ric|^{2}\nonumber\\ && -C|Ric|^{3}-C\frac{|\nabla Ric|}{|\nabla f|}|Ric|^{2}\nonumber\\
& \geq & |Rm|^{2}-c|Ric|^{2}-c\frac{|\nabla Ric|^{2}}{|\nabla f|^{2}}\nonumber\\&& +2|\nabla Ric|^{2}-C|Ric|^{3}-C\frac{|\nabla Ric|}{|\nabla f|}|Ric|^{2},
\end{eqnarray} 
where we have used again Proposition \ref{propA}.

Clearly,

\begin{equation*}
C\frac{|\nabla Ric|}{|\nabla f|}|Ric|^{2}\leq \frac{|\nabla Ric|^{2}}{|\nabla f|^{2}}+\frac {C^2}4|Ric|^{4}.
\end{equation*} Plugging this inequality into (\ref{ineqA12}) and using that the assumption $R\leq A+\varepsilon f$ implies $|\nabla f|^{2}\geq (1-\varepsilon)f-A,$ there exists a $r_{0}>0$ (sufficiently large) so that 

\begin{eqnarray}
\Delta_{f}\Big(|Rm|+|Ric|^{2}\Big) &\geq & |Rm|^{2}-c|Ric|^{3}-c|Ric|^{4}
\end{eqnarray} on $M\setminus D(r_{0}).$
It then follows that 
\begin{eqnarray*}
\Delta_{f}\Big(|Rm|+|Ric|^{2}\Big) &\geq & \frac{1}{2}\Big(|Rm|+|Ric|^{2}\Big)^{2}-c|Ric|^{4}-c|Ric|^{3}.
\end{eqnarray*} 
In addition, by using the estimate $|Ric| \leq c_{0}+ (c_{1}\varepsilon)\, f$ established in the first part of the proof, we see that

\begin{eqnarray}
\label{p1q32}
\Delta_{f}\left(|Rm|+|Ric|^{2}\right)\geq \frac{1}{2}\left(|Rm|+|Ric|^{2}\right)^{2}- \beta(f),
\end{eqnarray}  on $M\setminus D(r_{0}),$ where $\beta(f)=c\Big[(c_{0}+ (c_{1}\varepsilon)\, f)^{4}+(c_{0}+ (c_{1}\varepsilon)\, f)^{3}\Big].$

Now, as was accomplished in the proof of the first part, we consider the same functions $\eta$ and $\varphi (x)=\eta \Big(\frac{f(x)}{\rho^{2}}\Big)\in C_{0}^{\infty} (D(3\rho^2)),$ 
 for some arbitrary $\rho>2r_{0}$ sufficiently large.

Then one easily verifies that,  for $Q=\varphi^{2}\Phi$ with $\Phi=|Rm|+|Ric|^{2},$
\begin{eqnarray*}
\Delta_{f}Q&= &\Delta_{f}(\varphi^{2}\Phi)\nonumber\\&=& \Phi(\Delta_{f}\varphi^{2})+\varphi^{2}(\Delta_{f}\Phi)+2\langle \nabla \varphi^{2},\nabla \Phi\rangle,\nonumber
\end{eqnarray*} and taking into account that 
\begin{eqnarray*}
\langle \nabla \varphi^{2},\,\nabla Q\rangle &=& \Phi\langle \nabla \varphi^{2}, \nabla \varphi^{2}\rangle +\varphi^{2}\langle\nabla \varphi^{2},\nabla \Phi\rangle,
\end{eqnarray*}
 we therefore obtain
\begin{eqnarray}
\label{pkl11}
\varphi^{2}\Delta_{f}Q&= & \Big(\varphi^{2}\Phi\Big)\Delta_{f}\varphi^{2}+\varphi^{4}\Delta_{f}\Phi+2\varphi^{2}\langle \nabla \varphi^{2},\nabla \Phi\rangle\nonumber\\ &=& \Big(\varphi^{2}\Phi\Big)\Delta_{f}\varphi^{2}+\varphi^{4}\Delta_{f}\Phi +2\langle \nabla \varphi^{2},\nabla Q\rangle -2\Phi\langle \nabla \varphi^{2},\nabla \varphi^{2}\rangle.
\end{eqnarray} 
Thereby, by using $\Delta_{f}\varphi^{2}=2|\nabla \varphi|^{2}+2\varphi \Delta_{f}\varphi$ and (\ref{p1q32}), we infer

\begin{eqnarray}
\varphi^{2}\Delta_{f}Q &\geq & Q\left(2|\nabla \varphi|^{2}+2\varphi \Delta_{f}\varphi\right)+\varphi^{4}\left(\frac{1}{2}\Phi^{2}-\beta\right)\nonumber\\&&  +2\langle \nabla \varphi^{2},\nabla Q\rangle -8\Phi\varphi^{2}|\nabla \varphi|^{2} \nonumber\\ &\geq & \frac{1}{2}Q^{2}-cQ+2\langle \nabla \varphi^{2},\nabla Q\rangle -\beta.
\end{eqnarray} 
Therefore, by the standard maximum principle argument, we conclude that
$$ Q\leq c+ (c_{2}\varepsilon)\, f^{2}$$ on $D(3\rho^{2}).$ Since $Q=\varphi^{2}\Phi$ and $\rho>0$ is arbitrarily large, it holds that $$|Rm| \leq |Rm|+|Ric|^{2}\leq c+ (c_{2}\varepsilon)\, f^{2},$$ on $M\setminus D(r_{0})$.  Clearly, the constant $c$ can be chosen large enough so that (2) in Theorem 4 holds on all $M$. This finishes the proof of the theorem.
\end{proof}


\begin{bibdiv}
\begin{biblist}


\bib{caoALM11}{article}{
   author={Cao, Huai-Dong},
   title={Recent progress on Ricci solitons},
   conference={
      title={Recent advances in geometric analysis},
   },
   book={
      series={Adv. Lect. Math. (ALM)},
      volume={11},
      publisher={Int. Press, Somerville, MA},
   },
   date={2010},
   pages={1--38},
   review={\MR{2648937}},
}

\bib{CaoA}{article}{
   author={Cao, Huai-Dong},
   author={Chen, Bing-Long},
   author={Zhu, Xi-Ping},
   title={Recent developments on Hamilton's Ricci flow},
   conference={
      title={Surveys in differential geometry. Vol. XII. Geometric flows},
   },
   book={
      series={Surv. Differ. Geom.},
      volume={12},
      publisher={Int. Press, Somerville, MA},
   },
   date={2008},
   pages={47--112},
   review={\MR{2488948}},
}



\bib{CaoChen}{article}{
   author={Cao, Huai-Dong},
   author={Chen, Qiang},
   title={On Bach-flat gradient shrinking Ricci solitons},
   journal={Duke Math. J.},
   volume={162},
   date={2013},
   number={6},
   pages={1149--1169},
   issn={0012-7094},
   review={\MR{3053567}},
}


\bib{CC}{article}{
   author={Cao, Huai-Dong},
   author={Cui, Xin},
   title={Curvature estimates for four-dimensional gradient steady Ricci
   solitons},
   journal={J. Geom. Anal.},
   volume={30},
   date={2020},
   number={1},
   pages={511--525},
   issn={1050-6926},
   review={\MR{4058524}},
}

\bib{CZ}{article}{
   author={Cao, Huai-Dong},
   author={Zhou, Detang},
   title={On complete gradient shrinking Ricci solitons},
   journal={J. Differential Geom.},
   volume={85},
   date={2010},
   number={2},
   pages={175--185},
   issn={0022-040X},
   review={\MR{2732975}},
}

\bib{CH}{article}{
   author={Cao, Xiaodong},
   author={Tran, Hung},
   title={The Weyl tensor of gradient Ricci solitons},
   journal={Geom. Topol.},
   volume={20},
   date={2016},
   number={1},
   pages={389--436},
   issn={1465-3060},
   review={\MR{3470717}},
}


\bib{CWZ}{article}{
   author={Cao, Xiaodong},
   author={Wang, Biao},
   author={Zhang, Zhou},
   title={On locally conformally flat gradient shrinking Ricci solitons},
   journal={Commun. Contemp. Math.},
   volume={13},
   date={2011},
   number={2},
   pages={269--282},
   issn={0219-1997},
   review={\MR{2794486}},
}


\bib{Catino}{article}{
   author={Catino, Giovanni},
   title={Complete gradient shrinking Ricci solitons with pinched curvature},
   journal={Math. Ann.},
   volume={355},
   date={2013},
   number={2},
   pages={629--635},
   issn={0025-5831},
   review={\MR{3010141}},
}


\bib{catinoAdv}{article}{
   author={Catino, Giovanni},
   title={Integral pinched shrinking Ricci solitons},
   journal={Adv. Math.},
   volume={303},
   date={2016},
   pages={279--294},
   issn={0001-8708},
   review={\MR{3552526}},
}






\bib{Chan}{article}{
  author={Chan, Pak-Yeung},
   title={Curvature estimates for steady Ricci solitons},
   journal={Trans. Amer. Math. Soc.},
   volume={372},
   date={2019},
   number={12},
   pages={8985--9008},
   issn={0002-9947},
   review={\MR{4029719}},
}


\bib{Chen}{article}{
   author={Chen, Bing-Long},
   title={Strong uniqueness of the Ricci flow},
   journal={J. Differential Geom.},
   volume={82},
   date={2009},
   number={2},
   pages={363--382},
   issn={0022-040X},
   review={\MR{2520796}},
}


\bib{CW}{article}{
   author={Chen, Xiuxiong},
   author={Wang, Yuanqi},
   title={On four-dimensional anti-self-dual gradient Ricci solitons},
   journal={J. Geom. Anal.},
   volume={25},
   date={2015},
   number={2},
   pages={1335--1343},
   issn={1050-6926},
   review={\MR{3319974}},
}



\bib{CFSZ}{article}{author={Chow, Bennett}, author={Freedman, Michael}, author={Shin, Henry}, author={Zhang, Yongjia}, title={Curvature growth of some $4$-dimensional gradient Ricci soliton sigularity models}, journal={Advances in Mathematics}, date={2020}, number={372}, pages={107303}}


\bib{Chow}{article}{author={Chow, Bennet}, author={Lu, Peng}, author={Yang, Bo}, title={Lower bounds for the scalar curvatures of noncompact gradient Ricci solitons}, journal={C. R. Math. Acad. Sci. Paris}, volume={349}, date={2011}, number={23-24}, review={\MR{2861997}}, }



\bib{derd1}{article}{
   author={Derdzi\'{n}ski, Andrzej},
   title={Self-dual K\"{a}hler manifolds and Einstein manifolds of dimension
   four},
   journal={Compositio Math.},
   volume={49},
   date={1983},
   number={3},
   pages={405--433},
   issn={0010-437X},
   review={\MR{707181}},
}


\bib{ELM}{article}{
   author={Eminenti, Manolo},
   author={La Nave, Gabriele},
   author={Mantegazza, Carlo},
   title={Ricci solitons: the equation point of view},
   journal={Manuscripta Math.},
   volume={127},
   date={2008},
   number={3},
   pages={345--367},
   issn={0025-2611},
   review={\MR{2448435}},
}


\bib{Topping}{article}{
   author={Enders, Joerg},
   author={M\"{u}ller, Reto},
   author={Topping, Peter M.},
   title={On type-I singularities in Ricci flow},
   journal={Comm. Anal. Geom.},
   volume={19},
   date={2011},
   number={5},
   pages={905--922},
   issn={1019-8385},
   review={\MR{2886712}},
}


\bib{FIK}{article}{
   author={Feldman, Mikhail},
   author={Ilmanen, Tom},
   author={Knopf, Dan},
   title={Rotationally symmetric shrinking and expanding gradient
   K\"{a}hler-Ricci solitons},
   journal={J. Differential Geom.},
   volume={65},
   date={2003},
   number={2},
   pages={169--209},
   issn={0022-040X},
   review={\MR{2058261}},
}

\bib{FLGR}{article}{
   author={Fern\'{a}ndez-L\'{o}pez, Manuel},
   author={Garc\'{\i}a-R\'{\i}o, Eduardo},
   title={Rigidity of shrinking Ricci solitons},
   journal={Math. Z.},
   volume={269},
   date={2011},
   number={1-2},
   pages={461--466},
   issn={0025-5874},
   review={\MR{2836079}},
}




\bib{Hamilton1}{article}{author={Hamilton, Richard S.},
   title={Three-manifolds with positive Ricci curvature}, journal={J. Differential Geom.},, volume={17}, date={1982}, number={2}, pages={255--306},   review={\MR{664497}},
   }


\bib{Hamilton2}{article}{
   author={Hamilton, Richard S.},
   title={The formation of singularities in the Ricci flow},
   conference={
      title={Surveys in differential geometry, Vol. II},
      address={Cambridge, MA},
      date={1993},
   },
   book={
      publisher={Int. Press, Cambridge, MA},
   },
   date={1995},
   pages={7--136},
   review={\MR{1375255}},
}



\bib{Ivey}{article}{
   author={Ivey, Thomas},
   title={New examples of complete Ricci solitons},
   journal={Proc. Amer. Math. Soc.},
   volume={122},
   date={1994},
   number={1},
   pages={241--245},
   issn={0002-9939},
   review={\MR{1207538}},
}

\bib{KW}{article}{
   author={Kotschwar, Brett},
   author={Wang, Lu},
   title={Rigidity of asymptotically conical shrinking gradient Ricci solitons},
   journal={J. Differential Geom.},
   volume={100},
   date={2015},
   number={1},
   pages={55--108},
   review={\MR{3326574}},
}

\bib{MS}{article}{
   author={Munteanu, Ovidiu},
   author={Sesum, Natasa},
   title={On gradient Ricci solitons},
   journal={J. Geom. Anal.},
   volume={23},
   date={2013},
   number={2},
   pages={539--561},
   issn={1050-6926},
   review={\MR{3023848}},
}


\bib{MW}{article}{
   author={Munteanu, Ovidiu},
   author={Wang, Jiaping},
   title={Geometry of shrinking Ricci solitons},
   journal={Compos. Math.},
   volume={151},
   date={2015},
  number={12},
   pages={2273--2300},
   issn={0010-437X},
   review={\MR{3433887}},
}



\bib{MW2}{article}{
   author={Munteanu, Ovidiu},
   author={Wang, Jiaping},
   title={Positively curved shrinking Ricci solitons are compact},
   journal={J. Differential Geom.},
   volume={106},
   date={2017},
   number={3},
   pages={499--505},
   issn={0022-040X},
   review={\MR{3680555}},
}


\bib{Naber}{article}{
   author={Naber, Aaron},
   title={Noncompact shrinking four solitons with nonnegative curvature},
   journal={J. Reine Angew. Math.},
   volume={645},
   date={2010},
   pages={125--153},
   issn={0075-4102},
   review={\MR{2673425}},
}



\bib{Ni}{article}{
   author={Ni, Lei},
   author={Wallach, Nolan},
   title={On a classification of gradient shrinking solitons},
   journal={Math. Res. Lett.},
   volume={15},
   date={2008},
   number={5},
   pages={941--955},
   issn={1073-2780},
   review={\MR{2443993}},
}


\bib{Perelman2}{article}{author={Perelman, Grisha}, title={Ricci flow with surgery on three manifolds}, journal={ArXiv:math.DG/0303109}, date={2003},}


\bib{PW2}{article}{
   author={Petersen, Peter},
   author={Wylie, William},
   title={On the classification of gradient Ricci solitons},
   journal={Geom. Topol.},
   volume={14},
   date={2010},
   number={4},
   pages={2277--2300},
   issn={1465-3060},
   review={\MR{2740647}},

}




\bib{Sesum}{article}{
   author={Sesum, Natasa},
   title={Convergence of the Ricci flow toward a soliton},
   journal={Comm. Anal. Geom.},
   volume={14},
   date={2006},
   number={2},
   pages={283--343},
   issn={1019-8385},

}




\bib{Wu1}{article}{
   author={Wu, Peng},
   title={A Weitzenb\"{o}ck formula for canonical metrics on four-manifolds},
   journal={Trans. Amer. Math. Soc.},
   volume={369},
   date={2017},
   number={2},
   pages={1079--1096},
   issn={0002-9947},
   review={\MR{3572265}},
}


\bib{Wu}{article}{
   author={Wu, Jia-Yong},
   author={Wu, Peng},
   author={Wylie, William},
   title={Gradient shrinking Ricci solitons of half harmonic Weyl curvature},
   journal={Calc. Var. Partial Differential Equations},
   volume={57},
   date={2018},
   number={5},
   pages={Paper No. 141, 15},
   issn={0944-2669},
   review={\MR{3849152}},
}


\bib{zhang}{article}{
   author={Zhang, Zhu-Hong},
   title={Gradient shrinking solitons with vanishing Weyl tensor},
   journal={Pacific J. Math.},
   volume={242},
   date={2009},
   number={1},
   pages={189--200},
   issn={0030-8730},
   review={\MR{2525510}},
}


\bib{Zhang2}{article}{
   author={Zhang, Zhu-Hong}, title={A gap theorem of four-dimensional gradient shrinking solitons}, journal={Commun. Anal. Geom.}, volume={28}, date={2020}, number={3}, review={},  pages={729--742}, DOI = {10.4310/CAG.2020.v28.n3.a8}
   }




\end{biblist}
\end{bibdiv}
\end{document}